\documentclass[final,times]{elsarticle}
\usepackage{graphicx,subfigure,epsfig}
\usepackage{geometry}
\usepackage{amsmath,mathrsfs,amssymb,amsthm,dsfont,bm,amsfonts}         %
\usepackage{color}
\allowdisplaybreaks

\newtheorem{thm}{Theorem}[section]
\newtheorem{lem}{Lemma}[section]

\newcommand{\p}{\partial}
\newcommand{\f}{\frac}

\newcommand{\up}{\textup}
\newcommand{\tr}{\triangle}

\begin{document}

\begin{frontmatter}

\title{An optimal convergence analysis of the hybrid Raviart-Thomas mixed discontinuous Galerkin method for the Helmholtz equation
}

\author[zhang]{Jiansong Zhang}
\ead{jszhang@upc.edu.cn}

\author[zhu]{Jiang Zhu\corref{io}}
\ead{jiang@lncc.br}

\address[zhang]{Department of Applied Mathematics,
        China University of Petroleum, Qingdao 266580, China}
        
\address[zhu]{Laborat\'{o}rio Nacional de Computa\c{c}\~{a}o
Cient\'{i}fica, MCTI, Avenida Get\'{u}lio Vargas 333,
 25651-075 Petr\'{o}polis, RJ, Brazil}        

\cortext[io]{Corresponding author: Jiang Zhu.}

\begin{abstract}
The hybrid Raviart-Thomas mixed discontinuous Galerkin  (HRTMDG) method is proposed for solving the Helmholtz equation. With a new energy norm, we establish the existence and uniqueness of the HRTMDG method, and give its convergence analysis. The corresponding error estimate shows that the HRTMDG method has an optimal $L^2$-norm convergence accuracy which is independent of wavenumber.

\end{abstract}

\end{frontmatter}

\section{Introduction}

The Helmholtz boundary value problems with high wavenumber arise from many practical fields, such as, seismology, electromagnetics, underwater acoustics, medical imaging, and so on.  It is well known that traditional numerical methods do not work well for this problem and exhibit the so-called pollution effect \cite{BI1995}. 
To reduce or avoid the pollution effect, many researchers have done some various attempts, for example, \cite{ALCR08,WL2011,GM2011,FX2013,LS2017}. 

In this paper, we propose the HRTMDG method \cite{ES2010,ZZ2017,ZV2019} to solve Helmholtz equation, in which, the subproblems for the flux and unknown function are solved at element level and these variables are eliminated in favor of the Lagrange multiplier, identified as  the unknown function trace at the element interfaces, and the global system involves only the degrees of freedom associated with the multiplier, significantly reducing the computational cost. By the similar technique as in \cite{ES2010,ZZ2017},  with a new energy norm, we give the convergence analysis. The corresponding error estimate shows that the HRTMDG method has an optimal $L^2$-norm convergence accuracy which is independent of  wavenumber. 


\section{Formulation of HRTMDG method}
\setcounter{equation}{0}
Here we consider the following Helmholtz equation
\begin{equation}\label{eq1}
\begin{split}
\tr  u+\kappa^2u=\tilde{f}\quad &\textrm{in}\ \Omega,\\
u=g\quad &\textrm{on}\ \p\Omega,
\end{split}
\end{equation}
where $\Omega\subset\mathbb{R}^d$ $(d=2,3)$, a Lipschitz polyhedral domain with boundary $\p\Omega$; $\kappa>0$ is the wave number, $\tilde{f}\in L^2(\Omega)$ is the source term and boundary value $g\in H^{1/2}(\Omega)$.

By introducing an unknown variable $\bm\sigma=i \nabla u/\kappa$ and $f=i\tilde{f}/\kappa$,
 we can rewrite our problem into  an equivalent first order formulation:
\begin{equation}\label{eq2}
\begin{split}
i\kappa\bm\sigma+\nabla u=0\quad &\textrm{in}\ \Omega,\\
i\kappa u+\nabla\cdot\bm\sigma =f\quad &\textrm{in}\ \Omega,\\
u=g\quad &\textrm{on}\ \p\Omega.
\end{split}
\end{equation}
If $\kappa^2$ is not an eigenvalue for the above problem, there exists a unique solution $(\bm\sigma,u)\in H(\up{div};\Omega)\times H^1(\Omega)$ of \eqref{eq2} and this solution satisfies the standard elliptic regularity (see \cite{GM2011}):
\begin{equation}\label{eq3}
\|\bm\sigma\|_{H(\up{div};\Omega)}+\|u\|_{H^1(\Omega)}\leq C\{\|f\|_{L^2(\Omega)}+\|g\|_{H^{1/2}(\p(\Omega)}\}.
\end{equation}

In order to construct our procedure, we first give a partitioning of the domain $\Omega$. Let  $\mathcal{T}_{h}$  be a  regular division into simplices $K$. Let $\mathcal{E}_{h}$ $
=\{e:e$ is an edge of $K$ for all $K\in \mathcal{T}_{h}\}$, $\mathcal{E}_{h}^{i}$ $=\{e:e$ is an interior edge of $K\}$, $\mathcal{E}_{h}^{o }=$ $\mathcal{E }_{h}\cap \partial
\Omega $, and denote $\mathbf{n}$ the unit outward normal vector to the element boundary $e$. 

Define the inner products
\[
\begin{split}
(u,v)_K:=\int_K u\bar{v}dx,\quad (u,v) _{\mathcal{T}_h}:=\sum_{K\in\mathcal{T}_h}( u, v )_K,\\
\left\langle \lambda,\mu\right\rangle_e:=\int_e  \lambda\bar{\mu} ds,\quad \left\langle  \lambda,\mu\right\rangle_{\mathcal{E}_h}:=\sum_{e\in\mathcal{E}_h}\left\langle  \lambda,\mu\right\rangle_e,
\end{split}
\]
where the bar denotes complex conjugation. For vector-valued functions, some modification is required obviously. The corresponding norms are denoted by
\[
\|u\|^2_{\mathcal{T}_h}:=( u,u)_{\mathcal{T}_h},\quad
|\mu|^2_{\mathcal{E}_h}:=\left\langle  \mu,\mu\right\rangle_{\mathcal{E}_h}.
\]

Define the finite element spaces:
\[
\mathcal{V}_{h}=\{v_h\in L^2(\Omega): v_h|_{K}\in \mathbb{P}_k(K), \forall K\in \mathcal{T}%
_{h}\},\quad
\mathcal{W}_{h}=\{{\bm\tau}_h\in H(\up{div};\Omega): {\bm\tau}_h|_{K}\in RT_k(K), \forall
K\in \mathcal{T}_{h}\}
\]
where $\mathbb{P}_k(K)$ is the space of complex polynomials of degree $\leq k$ on element $K$, $RT_k(K)=  [\mathbb{P}_k(K)]^d\oplus  x\mathbb{P}_k(K)$ denotes the Raviart-Thomas mixed finite element space.
 So the traditional mixed variational formulation can be read as
 \begin{equation}\label{eq4}
\begin{array}{rcll}
(i\kappa \bm\sigma_h, {\bm\tau}_h)_{\mathcal{T}_h}-(u_h,\nabla\cdot{\bm\tau}_h)_{\mathcal{T}_h}&=&-(g,{\bm\tau}_h\cdot\mathbf{n})_{\mathcal{E}^o_h} \quad&{\bm\tau}_h\in \mathcal{W}_h,\\[.1in]
(i\kappa u_h,v_h)_{\mathcal{T}_h}+(\nabla\cdot\bm\sigma_h,v_h)_{\mathcal{T}_h}&=&(f,v_h)_{\mathcal{T}_h}\quad &v_h\in \mathcal{V}_h.
\end{array}
\end{equation}

 As we know,  the classical mixed finite element formulation \eqref{eq4} can keep the mass conservation on the discrete level, but it leads to a saddle-point problem and involves considerably more degree of freedoms than a standard $H^1$-conforming method. The HRTMDG method \cite{ES2010,ZZ2017,ZV2019} can overcome these problems.  Its main ideas are: by adding appropriate constraints, one can use completely discontinuous piecewise polynomial functions and ensure the continuity of the normal fluxes over element interfaces.
 
To give the HRTMDG procedure of the system \eqref{eq1}, we first introduce the piecewise Sobolev spaces given by
\[
\mathcal{H}^s(\mathcal{T}_{h})=\left\{z|_K\in {H}^s(K),\ \forall\ K\in \mathcal{T}_{h}\right\},\quad s\geq 0,\quad L^2(\mathcal{E}_h)=\{\mu\in L^2(e), \forall\ e\in \mathcal{E}_h\}.
\]
Redefine the finite element approximate spaces $\mathcal{W}_h$, $\mathcal{V}_h:$
\[
\mathcal{V}_h=\{v_h\in L^2(\mathcal{T}_h):\ v_h|_K\in  \mathbb{P}_k(K),\quad\forall\ K\in\mathcal{T}_{h}\},\quad \mathcal{W}_h=\{{\bm\tau}_h\in [\mathcal{H}^k(\mathcal{T}_h)]^d:\ {\bm\tau}_h|_K\in  RT_k(K),\quad \forall\ K\in\mathcal{T}_{h}\}
\]
and define the space for Lagrange multiplier
\[
\mathcal{M}_{h}=\{ \mu_h\in L^2(\mathcal{E }_{h}): \mu_h|_{e}\in \mathbb{P}_k(e), \forall e\in \mathcal{E}^i_{h},\ \mu_h|_{e}=0,\forall e\in \mathcal{E}^0_{h}\}
\]
where $\mathbb{P}_k(e)$ is the space of complex polynomials of degree $\leq k$ on edge $e$.

The corresponding HRTMDG finite element problem can be written as follows:
 \begin{equation}\label{eq5}
\begin{array}{rcll}
(i\kappa \bm\sigma_h, {\bm\tau}_h)_{\mathcal{T}_h}-(u_h,\nabla\cdot{\bm\tau}_h)_{\mathcal{T}_h}+\langle \lambda_h,{\bm\tau}_h\cdot\mathbf{n}\rangle_{\mathcal{E}_h}&=& -(g,{\bm\tau}_h\cdot\mathbf{n})_{\mathcal{E}^o_h} \quad&{\bm\tau}_h\in \mathcal{W}_h,\\[.1in]
-(i\kappa u_h,v_h)_{\mathcal{T}_h}-(\nabla\cdot\bm\sigma_h,v_h)_{\mathcal{T}_h}&=&-(f,v_h)_{\mathcal{T}_h}\quad &v_h\in \mathcal{V}_h,\\[.1in]
\langle  \bm\sigma_h\cdot\mathbf{n}, \mu_h\rangle_{\mathcal{E}_h}&=&0\quad&\mu_h\ \in \mathcal{M}_h.
\end{array}
\end{equation}

By integrating by parts, we have

\noindent
{\bf HRTMDG method.} Find $(\bm\sigma_h,u_h,\lambda_h)\in\mathcal{W}_h\times\mathcal{V}_h\times \mathcal{M}_h$ such that
  \begin{equation}\label{eq6}
\mathcal{A}(\bm\sigma_h,u_h,\lambda_h;{\bm\tau}_h,v_h,\mu_h)=\mathcal{F}({\bm\tau}_h,v_h,\mu_h),\quad ({\bm\tau}_h,v_h,\mu_h)\in \mathcal{W}_h\times\mathcal{V}_h\times \mathcal{M}_h,
\end{equation}
where $\mathcal{A}$ and $\mathcal{F}$ are defined by
\[
\begin{split}
\mathcal{A}(\bm\sigma_h,u_h,\lambda_h;{\bm\tau}_h,v_h,\mu_h)
:=&(i\kappa \bm\sigma_h, {\bm\tau}_h)_{\mathcal{T}_h}
-(i\kappa u_h,v_h)_{\mathcal{T}_h}
+(\bm\sigma_h,\nabla v_h)_{\mathcal{T}_h}
+(\nabla u_h,{\bm\tau}_h)_{\mathcal{T}_h}
\\&+\langle \lambda_h-u_h,{\bm\tau}_h\cdot\mathbf{n}\rangle_{\mathcal{E}_h}
+\langle  \bm\sigma_h\cdot\mathbf{n}, \mu_h-v_h\rangle_{\mathcal{E}_h}
\end{split}
\]
and
\[
\mathcal{F}({\bm\tau}_h,v_h,\mu_h)= -(f,v_h)_{\mathcal{T}_h}-(g,{\bm\tau}_h\cdot\mathbf{n})_{\mathcal{E}^o_h}.
\]

On every element, we can write (\ref{eq6}) into a matrix
equation of the form
\[
\left(
\begin{array}{ccc}
\mathbf{A} & \mathbf{B} & \mathbf{D} \\
\mathbf{B}^t & \mathbf{E} & 0 \\
\mathbf{D}^t & 0 & 0 \\
\end{array}
\right) \left(
\begin{array}{c}
\bm{\sigma}_h \\
u_h \\
{\lambda}_h \\
\end{array}
\right) = \left(
\begin{array}{c}
\mathbf{F}_1 \\
\mathbf{F}_2 \\
0 \\
\end{array}
\right),
\]
where $\mathbf{F}_1$ incorporates the Dirichlet boundary data and $\mathbf{F}_2$ is the vector with respect to the right-hand side $f$. The local matrix $\mathbf{A,B,D,E}$ can be computed by their corresponding finite
element basis. $\mathbf{D}^t$ and $\mathbf{B}^t$  denote the complex conjugate transpose matrices of $\mathbf{D}$ and $\mathbf{B}$, respectively. Both the vectors of $\bm{\sigma}_h$ and $u_h$ can now be easily eliminated to obtain an equation for the multiplier only, namely,
\[
\mathbf{D}^t\mathbf{M}^{-1}\mathbf{D}\bm{\lambda_h}=\mathbf{G},
\]
where $\mathbf{M}$ and $\mathbf{G}$ are given by
\[
\mathbf{M}=\mathbf{A}-\mathbf{B}\mathbf{E}^{-1}\mathbf{B}^t,  \quad
\mathbf{G}=\mathbf{D}^t\mathbf{M}^{-1}(\mathbf{F}_1-\mathbf{B}\mathbf{E}^{-1}\textbf{F}_2).   
\]
Assembling the above equation on every elements, we can get the global
systems for solving the multiplier ${\lambda}_h$.

From the above, we can see that the HRTMDG method has several
advantages: (I) compared with the discontinuous Galerkin finite element method, the number of degrees of freedom of multiplier is remarkably
small; (II) once the multiplier ${\lambda}_h$ has been obtained, $\bm{\sigma}_h$ and $u_h$ can be efficiently computed element by element; (III) the matrix $\mathbf{D}^t\mathbf{M}^{-1}\mathbf{D}$ is symmetric and positive definite,
so we can solve the systems by using the conjugate gradient method.

\begin{thm}\label{thm1}
(Consistency) 
HRTMDG method (\ref{eq6})  is consistent. That is,  let $u$ be the solution of \eqref{eq1},  $\bm\sigma=i\nabla u/\kappa$, and $\lambda=u$. Then the variational equation \eqref{eq6} holds if $\bm\sigma_h$, $u_h$ and $\lambda_h$ are replaced by  $\bm\sigma$, $u$ and $\lambda$. 
\end{thm} 
 \begin{proof}
 Let $u$ denote the solution of \eqref{eq1}, and make substitutions as mentioned in Theorem \ref{thm1}. Taking $({\bm\tau}_h,v_h,\mu_h)=({\bm\tau}_h,0,0)$ in \eqref{eq6}, we can get
\[
\mathcal{A}(\f{i}{\kappa}\nabla u,u,u;{\bm\tau}_h,0,0)
=-(\nabla u, {\bm\tau}_h)_{\mathcal{T}_h}+(\nabla u,{\bm\tau}_h)_{\mathcal{T}_h}
-\langle u,{\bm\tau}_h\cdot\mathbf{n}\rangle_{\mathcal{E}^o_h}
=
-\langle u,{\bm\tau}_h\cdot\mathbf{n}\rangle_{\mathcal{E}^o_h}=-(g,{\bm\tau}_h\cdot\mathbf{n})_{\mathcal{E}^o_h}.
\]

Next, choosing $({\bm\tau}_h,v_h,\mu_h)=(0,v_h,0)$ in \eqref{eq6} and using Green's formula, we have
\[
\mathcal{A}(\f{i}{\kappa}\nabla u,u,u;0,v_h,0)
=-(i\kappa u,v_h)_{\mathcal{T}_h}
+(\f{i}{\kappa}\nabla u,\nabla v_h)_{\mathcal{T}_h}
-\langle \f{i}{\kappa}\nabla u\cdot\mathbf{n}, v_h\rangle_{\mathcal{E}_h}
=-\f{i}{\kappa}(\tr  u+\kappa^2 u,v_h)_{\mathcal{T}_h}
=-\f{i}{\kappa}(\tilde{f},v_h)_{\mathcal{T}_h},
\]
where we have used the fact that $u$ is the solution of \eqref{eq1} in the last equation.

Finally, testing with $ ({\bm\tau}_h,v_h,\mu_h)=(0,0,\mu_h)$, we can obtain
\[
\mathcal{A}(\f{i}{\kappa}\nabla u,u,u;0,0,\mu_h)
=\langle  \f{i}{\kappa}\nabla u\cdot\mathbf{n}, \mu_h\rangle_{\mathcal{E}_h}=0,
\]
which implies that the normal flux $\nabla u\cdot\mathbf{n}$ is continuous across element interfaces.
 \end{proof}
 
 \begin{thm}\label{thm2}
(Conservation) 
HRTMDG method (\ref{eq6}) is locally and globally conservative. 
\end{thm} 
\begin{proof}
Let $\chi_K$ denote the characteristic function of a set $K\subset{\Omega}$. Taking $ (\omega_h,v_h,\mu_h)=(0,\chi_K,0)$ in \eqref{eq6}, we can get
\[
\begin{split}
-(i\kappa u_h,1)_{K}
-\sum_{e\in\p K}\langle  \bm\sigma_h\cdot\mathbf{n}, 1\rangle_{e}=-(f,1)_K,
\end{split}
\]
the above equation implies that HRTMDG method (\ref{eq6}) keeps local mass balance, and hence it also keeps globally mass balance.
\end{proof}

\section{Existence and uniqueness}

In order to give the existence and uniqueness of HRTMDG method (\ref{eq6}), we first define the following mesh-dependent energy norm
\[
|\|({\bm\tau}, v,\mu)\||_{A}:=\left(\kappa\|{\bm\tau}\|^2_{\mathcal{T}_h}+\kappa \|u\|^2_{\mathcal{T}_h} +\f{1}{\kappa}\|\nabla u\|^2_{\mathcal{T}_h}+\f{1}{\kappa h}|\mu-v|^2_{\mathcal{E}_h}\right)^{\f{1}{2}}.
\] 
 
Next, we show the stability and boundedness of the bilinear form $\mathcal{A}(\cdot;\cdot)$ in the sense of energy norms $|\|\cdot\||_{A}$. The following result ( see Lemma 3.1 in \cite{ES2010} ) will be used.
\begin{lem}\label{lem1}
For given $v_h\in \mathcal{V}_h$ and $\mu_h\in\mathcal{M}_h$, there exists a unique solution $\tilde{{\bm\tau}}\in \mathcal{W}_h$ such that
\begin{equation}\label{eq7}
\begin{split}
&(\tilde{{\bm\tau}},p)_K=(\nabla v_h, p)_K \quad\forall p\in [\mathbb{P}_{k-1}(K)]^d,\\
&\langle \tilde{{\bm\tau}}\cdot \mathbf{n}, q\rangle_{\p K}= \langle \mu_h, q\rangle_{\p K} \quad\forall q\in\mathbb{P}_{k}(\p K),
\end{split}
\end{equation}
and the following estimate holds
\begin{equation}\label{eq8}
\|\tilde{{\bm\tau}}\|_{\mathcal{T}_h}\leq c_I\left(\|\nabla v_h\|^2_{\mathcal{T}_h}+h|\mu_h|^2_{\mathcal{E}_h}\right)^{\f{1}{2}},
\end{equation}
where  $c_I$ is  a constant independent of mesh size $h$ and $\kappa$.
\end{lem}

\begin{lem}\label{thm3}
(Stability) There exists a positive constant $c_A$  that is independent of the mesh size $h$  and $\kappa$ such that 
\begin{equation}\label{eq9}
\sup_{({\bm\tau}_h,v_h,\mu_h)}
\f{\left|\mathcal{A}(\bm\sigma_h,u_h,\lambda_h;{\bm\tau}_h,v_h,\mu_h)\right|}
{|\|({\bm\tau}_h,v_h,\mu_h)\||_{A}}
\geq c_A|\|(\bm\sigma_h,u_h,\lambda_h)\||_{A}
\end{equation}
holds for all $(\bm\sigma_h,u_h,\lambda_h)\in\mathcal{W}_h\times\mathcal{V}_h\times\mathcal{M}_h$.
\end{lem}
\begin{proof}
Taking $v_h=\f{1}{\kappa}u_h$ and $\mu_h=\f{1}{\kappa h}(\lambda_h-u_h)$ in \eqref{eq7}, we can obtain

\begin{equation}\label{eq10}
\begin{split}
(\tilde{{\bm\tau}},p)_K&=( \f{1}{\kappa}\nabla u_h, p)_K,\\
\langle \tilde{{\bm\tau}}\cdot \mathbf{n}, q\rangle_{\p K}&= \langle \f{1}{\kappa h}(\lambda_h-u_h), q\rangle_{\p K},
\end{split}
\end{equation}
and
\begin{equation}\label{eq11}
\|\tilde{{\bm\tau}}\|_{\mathcal{T}_h}\leq c_I\left(\f{1}{\kappa^2}\|\nabla u_h\|^2_{\mathcal{T}_h}+\f{1}{\kappa^2 h}|\lambda_h-u_h|^2_{\mathcal{E}_h}\right)^{\f{1}{2}}.
\end{equation}
For any $\epsilon>0$, we have
\[
\begin{split}
\mathcal{A}(\bm\sigma_h,u_h,\lambda_h;\epsilon\tilde{{\bm\tau}},0,0)
=&\epsilon(i\kappa \bm\sigma_h, \tilde{{\bm\tau}})_{\mathcal{T}_h}
+\epsilon(\nabla u_h,\tilde{{\bm\tau}})_{\mathcal{T}_h}
+\epsilon\langle \lambda_h-u_h,\tilde{{\bm\tau}}\cdot\mathbf{n}\rangle_{\mathcal{E}_h}\\
=&\epsilon(i\kappa \bm\sigma_h, \tilde{{\bm\tau}})_{\mathcal{T}_h}
+\epsilon\left(\f{1}{\kappa}\|\nabla u_h\|^2_{\mathcal{T}_h}
+\f{1}{\kappa h}| \lambda_h-u_h|^2_{\mathcal{E}_h})\right)\\
\geq &-\f{\kappa}{2}\|\bm\sigma_h\|^2_{\mathcal{T}_h}-\f{\kappa\epsilon^2}{2}\|\tilde{{\bm\tau}}\|^2_{\mathcal{T}_h}+\epsilon\left(\f{1}{\kappa}\|\nabla u_h\|^2_{\mathcal{T}_h}
+\f{1}{\kappa h}| \lambda_h-u_h|^2_{\mathcal{E}_h})\right)
\\\geq&-\f{\kappa}{2}\|\bm\sigma_h\|^2_{\mathcal{T}_h}+\left(\epsilon-\f{c_I^2\epsilon^2}{2}\right)\left(\f{1}{\kappa}\|\nabla u_h\|^2_{\mathcal{T}_h}
+\f{1}{\kappa h}| \lambda_h-u_h|^2_{\mathcal{E}_h})\right)
\end{split}
\]
where we have used \eqref{eq10} in the second equation and \eqref{eq11} in the second inequality. Set $\epsilon=1/c_I$. From \eqref{eq11} we have 
\begin{equation}\label{eq12}
\mathcal{A}(\bm\sigma_h,u_h,\lambda_h;\epsilon\tilde{{\bm\tau}},0,0)
\geq-\f{\kappa}{2}\|\bm\sigma_h\|^2_{\mathcal{T}_h}+\f{c_I}{2}\left(\f{1}{\kappa}\|\nabla u_h\|^2_{\mathcal{T}_h}
+\f{1}{\kappa h}| \lambda_h-u_h|^2_{\mathcal{E}_h})\right).
\end{equation}
Note that
\begin{equation}\label{eq13}
\mathcal{A}(\bm\sigma_h,u_h,\lambda_h;i\bm\sigma_h,-iu_h,-i\lambda_h)
=\kappa \|\bm\sigma_h\|^2_{\mathcal{T}_h}
+\kappa\| u_h\|^2_{\mathcal{T}_h}.
\end{equation}
Combining the above two inequalities for the two choices of test functions, we complete our proof of Lemma \ref{thm3}.
\end{proof}

\begin{lem}\label{thm4}
(Boundedness) There exists a constant $C_A$ independent of $h$  and $\kappa$, such that, 
\begin{equation}\label{eq14}
|\mathcal{A}(\bm\sigma_h,u_h,\lambda_h;{\bm\tau}_h,v_h,\mu_h)|\leq C_A|\|(\bm\sigma_h,u_h,\lambda_h)\||_{A}|\|({\bm\tau}_h,v_h,\mu_h)\||_{A}
\end{equation}
holds for all $(\bm\sigma_h,u_h,\lambda_h), ({\bm\tau}_h,v_h,\mu_h)\in\mathcal{W}_h\times\mathcal{V}_h\times\mathcal{M}_h$.
\end{lem}
\begin{proof}
Using the standard arguments, we can easily get the above estiamte.
\end{proof}

By the stability and boundedness of the bilinear form $\mathcal{A}(\cdot,\cdot)$, and Lax-Milgram theorem, we obtain the following result.
\begin{thm}
HRTMDG method \eqref{eq6} has a unique solution.
\end{thm}

\section{A uniform error estimate}
For $K\in\mathcal{T}_h$, $e\in\mathcal{E}_h$ and functions $u\in L^2(K)$ and $\lambda\in L^2(e)$, we define the local $L^2$-projection operaters $\Pi^K$ and $\Pi_e$ by
\[
\left\langle u-\Pi^Ku,v_h\right\rangle_{K}=0,\quad\forall v_h\in \mathbb{P}_{k}(K)
\]
and
\[
\left\langle \lambda-\Pi^e\lambda,\mu_h\right\rangle_{e}=0,\quad\forall \mu_h\in \mathbb{P}_{k}(e).
\]
The following error estimates hold:
\begin{equation}\label{eq15}
\begin{array}{rcll}
  \|u-\Pi^Ku\|_{K}& \leq  &  C^{*}h^s|u|_{s,K}, & 0\leq s\leq k+1,\\
  \|\nabla(u-\Pi^Ku)\|_{K}& \leq  & C^{*}h^s|u|_{s+1,K}, & 0\leq s\leq k,  \\
   \|u-\Pi^Ku\|_{\p K}+ \|u-\Pi^e u\|_{\p K}& \leq &   C^{*}h^{s+1/2}|u|_{s+1,K}, & 0\leq s\leq k
\end{array}
\end{equation}
where $C^{*}$ is a constant independent of $h$ and $\kappa$.

Similarly, the interpolation operators for functions on $\mathcal{T}_h$ and $\mathcal{E}_h$ are defined element-wise and are denoted by the same symbols.

We also introduce the Raviart-Thomas projection operator (see \cite{BF1991}) such that
\begin{equation}\label{eq16}
\begin{array}{rcll}
  (\bm\sigma-\Pi^{RT}\bm\sigma,p_h)_K& =  &0, &\forall p_h\in[\mathbb{P}_{k-1}(K)]^d ,  \\
 ((\bm\sigma-\Pi^{RT}\bm\sigma)\cdot\mathbf{n},\mu_h)_e&=   & 0, &\forall\mu_h\in \mathbb{P}_{k}(e),\quad e\in\p K.
\end{array}
\end{equation}
It is well known that, there exists a constant $C^{*}$ independent of $h$ and $\kappa$ such that the following approximate properties:
\begin{equation}\label{eq17}
\begin{array}{rcll}
  \|\bm\sigma-\Pi^{RT}\bm\sigma\|_{K}+ h^{1/2}\|\bm\sigma-\Pi^{RT}\bm\sigma\|_{\p K}& \leq  &  C^{*}h^s|\bm\sigma|_{s,K}, & 1/2\leq s\leq k+1,\\
  \|\nabla\cdot(\bm\sigma-\Pi^{RT}\bm\sigma)\|_{K}& \leq  & C^{*}h^s|\nabla\cdot\bm\sigma|_{s,K}, & 1\leq s\leq k+1.
\end{array}
\end{equation}

As we know, the error of HRTMDG method (\ref{eq6}) can be divided into two parts:  an approximate error and a discrete error. We need to estimate the discrete error.

\begin{thm}\label{thm5}
Let $(u,\bm\sigma)$  and $(\bm\sigma_h,u_h,\lambda_h)$ be the solutions of \eqref{eq2} and \eqref{eq6}, respectively. Then there exists a constant $C$ independent of $h$ and $\kappa$  such that
\begin{equation}\label{eq18}
|\|(\Pi^{RT}\bm\sigma-\bm\sigma_h,\Pi^Ku-u_h,\Pi^e u-\lambda_h)\||_{A}
\leq C\sqrt{\kappa} \|\Pi^{RT}\bm\sigma-\bm\sigma\|_{\mathcal{T}_h}.
\end{equation}
\end{thm}
\begin{proof}
Using the stability of the bilinear form and Galerkin orthogonality, we have \begin{equation}\label{eq19}
\begin{split}
c_A|\|(\Pi^{RT}\bm\sigma-\bm\sigma_h,\Pi^Ku-u_h,\Pi^e u-\lambda_h)\||_{A}
\leq &\sup_{({\bm\tau}_h,v_h,\mu_h)}
\f{\left|\mathcal{A}(\Pi^{RT}\bm\sigma-\bm\sigma_h,\Pi^Ku-u_h,\Pi^e u-\lambda_h;{\bm\tau}_h,v_h,\mu_h)\right|}
{|\|({\bm\tau}_h,v_h,\mu_h)\||_{A}}
\\= &\sup_{({\bm\tau}_h,v_h,\mu_h)}
\f{\left|\mathcal{A}(\Pi^{RT}\bm\sigma-\bm\sigma,\Pi^Ku-u,\Pi^eu- u;{\bm\tau}_h,v_h,\mu_h)\right|}{|\|({\bm\tau}_h,v_h,\mu_h)\||_{A}}.
\end{split}
\end{equation}

In fact, utilizing the definitions of the bilinear form $\mathcal{A}$ and projection operators, we can get
\begin{equation}\label{eq21}
\begin{split}
& \mathcal{A}(\Pi^{RT}\bm\sigma-\bm\sigma,\Pi^Ku-u,\Pi^eu- u;{\bm\tau}_h,v_h,\mu_h)
\\&\quad=(i\kappa (\Pi^{RT}\bm\sigma-\bm\sigma), {\bm\tau}_h)_{\mathcal{T}_h}
-(i\kappa (\Pi^Ku-u),v_h)_{\mathcal{T}_h}
+(\Pi^{RT}\bm\sigma-\bm\sigma,\nabla v_h)_{\mathcal{T}_h}
+(\nabla (\Pi^Ku-u),{\bm\tau}_h)_{\mathcal{T}_h}
\\&\quad\quad+\langle (\Pi^eu- u)- (\Pi^Ku-u),{\bm\tau}_h\cdot\mathbf{n}\rangle_{\mathcal{E}_h}
+\langle(\Pi^{RT}\bm\sigma-\bm\sigma)\cdot\mathbf{n}, \mu_h-v_h\rangle_{\mathcal{E}_h}
\\&\quad=(i\kappa (\Pi^{RT}\bm\sigma-\bm\sigma), {\bm\tau}_h)_{\mathcal{T}_h}
-(i\kappa (\Pi^Ku-u),v_h)_{\mathcal{T}_h}
-(\nabla\cdot(\Pi^{RT}\bm\sigma-\bm\sigma), v_h)_{\mathcal{T}_h}-(\Pi^Ku-u,\nabla\cdot{\bm\tau}_h)_{\mathcal{T}_h}
\\&\quad\quad+\langle \Pi^eu- u,{\bm\tau}_h\cdot\mathbf{n}\rangle_{\mathcal{E}_h}
+\langle(\Pi^{RT}\bm\sigma-\bm\sigma)\cdot\mathbf{n}, \mu_h\rangle_{\mathcal{E}_h}
=(i\kappa (\Pi^{RT}\bm\sigma-\bm\sigma), {\bm\tau}_h)_{\mathcal{T}_h}.
\end{split}
\end{equation}
Substituting \eqref{eq21} into \eqref{eq19}, we  get the estimate \eqref{eq18}.
\end{proof}
 
So, we obtain the following optimal a priori error estimate.
\begin{thm}\label{thm7}
Let $(u,\bm\sigma)$  and $(\bm\sigma_h,u_h,\lambda_h)$ be the solutions of \eqref{eq2} and \eqref{eq6}, respectively. Then there exists a constant $C$ independent of the mesh size $h$ and $\kappa$ such that 
\begin{equation}\label{eq23}
\|\bm\sigma-\bm\sigma_h\|_{\mathcal{T}_h}+\|u-u_h\|_{\mathcal{T}_h}
\leq Ch^s\{|\bm\sigma|_{s,\mathcal{T}_h}+ |u|_{s,\mathcal{T}_h}\},\quad 1/2\leq s\leq k+1.
\end{equation}
\end{thm}
\begin{proof}
Using the definition of the norm $|\|\cdot\||_{A}$ and the estimate \eqref{eq18}, we get
\[
\|\Pi^{RT}\bm\sigma-\bm\sigma_h\|_{\mathcal{T}_h}+\|\Pi^Ku-u_h\|_{\mathcal{T}_h}
\leq C\|\Pi^{RT}\bm\sigma-\bm\sigma\|_{\mathcal{T}_h}.
\]
Combining the above estimate with the approximate properties \eqref{eq15} and \eqref{eq17}, we obtain the estimate \eqref{eq23}.
\end{proof}

\section*{Acknowledgments}
J. Zhang's work was supported partially by the Major Scientific and Technological Projects of CNPC under Grant (ZD2019-183-008), the Natural Science Foundation of Shandong Province (ZR2019MA015) and the Fundamental Research Funds for the Central Universities  (20CX05011A). J. Zhu's work was supported partially by the National Council for Scientific and Technological Development (CNPq).



\section*{References}
\begin{thebibliography}{00}
\bibitem{BI1995}
I. Babuska, F. Ihlenburg, E.T. Paik, S.A. Sauter,
A generalized finite element method for solving the Helmholtz equation in two dimensions with minimal pollution, Comput. Methods Appl. Mech.
Engrg. 128 (1995) 325-359.

\bibitem{ALCR08}
E.G.D. Carmo, G. B. Alvarez, A.F.D. Loula, F. A. Rochinha, 
A nearly optimal Galerkin projected residual finite element method for Helmholtz problem, Comput. Methods Appl. Mech. Engrg. 197 (2008) 1362-1375.

\bibitem{WL2011}
Y.S. Wong, G. Li,
Exact finite difference schemes for solving Helmholtz equation at any wavenumber,
Inter. J. Numer. Anal. Model. Ser. B, 2 (2011) 91-108.
 
\bibitem{GM2011} 
R. Griesmaier, P. Monk,
Error analysis for a hybridizable discontinuous Galerkin method for the Helmholtz equation, J. Sci. Comput. 49 (2011) 291-310.

\bibitem{FX2013}
X. Feng, Y. Xing, Absolutely stable local discontinuous Galerkin methods for the Helmholtz equation with large wave number, Math. Comp. 82 (2013)
1269-1296.

\bibitem{LS2017}
 C.Y. Lama, C.-W. Shu, A phase-based interior penalty discontinuous Galerkin method for the Helmholtz equation with spatially varying wavenumber, Comput. Methods Appl. Mech. Engrg., 318 (2017) 456-473.

\bibitem{ES2010}
H. Egger, J. Schoberl,
A hybrid mixed discontinuous Galerkin finite-element method for convection-diffusion problems, 
IMA J. Numer. Anal. 30 (2010) 1206-1234.

\bibitem{ZZ2017}
J. Zhang, J. Zhu, R. Zhang, D. Yang, A.F.D. Loula,
A combined discontinuous Galerkin finite element method for miscible displacement problem,
J. Comp. Appl. Math. 309 (2017) 44-55.

\bibitem{ZV2019}
J. Zhu, H. A. Vargas P., 
Robust and efficient mixed hybrid discontinuous finite element methods for elliptic interface problems,
Inter. J. Numer. Anal. Model. 16 (2019) 767-788.

\bibitem{BF1991}
F.  Brezzi, M. Fortin, Mixed and Hybrid Finite Element Methods, New York: Springer, 1991.

\end{thebibliography}
\end{document}